\documentclass[12pt, authoryear]{elsarticle}
\usepackage{amssymb,amsmath}
\usepackage{amsthm,color}
\usepackage{mathptmx}
\usepackage[a4paper,margin=2cm]{geometry}
\theoremstyle{plain}
\newtheorem{lemma}{Lemma}

\newtheorem{theorem}{Theorem}
\theoremstyle{remark}
\newtheorem{remark}{Remark}

\newcommand{\R}{\mathbb{R}}

\newcommand{\ci}{\mathfrak{i}}

\newcommand{\dif}{\mathsf{d}}
\newcommand{\lp}{\mathsf L}

\newcommand{\EX}{\mathbf{E}}

\journal{arXiv}

\begin{document}
\begin{frontmatter}
\title{The Order Barrier for Strong Approximation \\ of Rough Volatility Models}
\author[label1]{Andreas Neuenkirch}%
\ead{neuenkirch@kiwi.math.uni-mannheim.de}
\cortext[cor1]{Corresponding author}
\author[label1]{Taras Shalaiko\corref{cor1}}
\ead{tshalaik@uni-mannheim.de}
\address[label1]{University of Mannheim, Institute of Mathematics, A5,6, Mannheim, D-68131, Germany}

\begin{abstract} 
We study the strong approximation of a rough volatility model, in which the log-volatility  is  given by a fractional Ornstein-Uhlenbeck process with Hurst parameter $H<1/2$.   Our methods are based on 
an equidistant  discretization  of the volatility process and of the
driving Brownian motions, respectively. 
For the root mean-square error at a single point  the  optimal
rate of convergence that can be achieved by such methods is $n^{-H}$, where $n$ denotes the number of subintervals of the discretization. This rate is in particular obtained by the Euler method and
an Euler-trapezoidal type scheme.
\end{abstract}
\end{frontmatter}

\section{Introduction and Main Results}
Let $B=\{ B_t, t \in \mathbb{R}\}$ be a fractional Brownian motion (fBm in what follows) with Hurst parameter $H \in (0,1/2)$, i.e. $B$ is a centered
Gaussian processes with continuous sample paths,  $B_0=0$ and mean square smoothness
$$ \EX |B_t-B_s|^2= |t-s|^{2H}, \quad s,t \in \mathbb{R}.$$
Moreover, let $V=\{ V_t, t \geq 0\}$,  $W=\{ W_t, t \geq 0\}$ be two independent  Brownian motions, $\mu \in \mathbb{R}, \lambda, \theta, s_0 >0$, $\rho \in (-1,1)$ and consider 
\begin{align}   \nonumber
 S_t &=s_0 e^{X_t}, \\ 
X_t &= - \frac{1}{2} \int_0^t e^{2 Y_s} \dif s +  \rho \int_0^{t} e^{Y_s}\dif V_s + \sqrt{1-\rho^2} \int_0^{t} e^{Y_s} \dif W_s,\\
 Y_t &=  \mu  +  \theta e^{-\lambda t }\int_{-\infty}^t e^{\lambda s} \dif B_s. \nonumber
\end{align} 
Here $X=\{ X_t, t \geq 0 \}$ models the log-price of an asset, whose log-volatility $Y=\{ Y_t, t \geq 0 \}$ is given by the stationary solution of the Langevin equation
$$ \dif Y_t = \lambda( \mu -Y_t ) \dif t + \theta \dif B_t.$$ The
 fractional Brownian motion $B$ and the Brownian motion $V$ are correlated, i.e.
$$ \EX B_t V_s =  \gamma(t,s), \qquad   t \in \mathbb{R}, \, s \geq 0,$$
for some suitable, i.e. in  particular positive definite, function $\gamma: \R \times [0, \infty) \rightarrow \mathbb{R}$, while $B$ and $W$ are independent, i.e.
$$ \EX B_t W_s =0, \qquad t \in \mathbb{R}, \, s \geq 0.$$
Such a model has been proposed by Gatheral, Jaisson and Rosenbaum based on striking empirical evidence
that  the log-volatility of assets behaves essentially as  fBm with with $H \approx 0.1$, see \cite{GJR}. This model has been further analysed in \cite{BFG}.

In this manuscript, we will study the optimal mean square approximation of $X_T$ based on \begin{align} V_0, V_{T/n}, \ldots, V_T,\, \, \,
W_0, W_{T/n}, \ldots, W_T, \, \, \,  Y_{0},Y_{T/n}, \ldots, Y_T. \label{info}
\end{align} The fractional Ornstein-Uhlenbeck process (fOUp) $Y$ is a Gaussian process with known mean covariance function and thus exact joint simulation of $V,W,Y$ at a finite number of time points is
possible. Clearly, the optimal mean square approximation of $X_T$ using \eqref{info} is given by
\begin{align} X_n^{\sf opt}=\EX \big ( \, X_T  \,  \big | \, V_{kT/n}, W_{kT/n}, Y_{kT/n}, \, k=0, \ldots, n \, \big )\end{align}
and the corresponding minimal errors are
\begin{align} e(n)=  \left( \EX|X_T-  X_n^{\sf opt}|^2 \right)^{1/2} . \end{align}

We will show that rough volatility models are numerically tough in the sense that they admit only low convergence rates for the mean square approximation based on the information given by \eqref{info}. More precisely, we will show that
\begin{align*}
\liminf_{n \rightarrow \infty} \,  \left(\frac{n}{T}\right)^{2H}  e(n)^2 \geq (1-\rho^2) \frac{1}{(2H+1)(2H+2)}  T \theta^2  {\bf E}|e^{Y_0}|^2 ,
\end{align*}
see Theorem \ref{lower_b}. Moreover, the optimal convergence rate $n^{-H}$ is obtained by the Euler method
\begin{align}
\label{Euler}
X_n^{ {\sf E}} &= - \frac{1}{2}  \sum_{k=0}^{n-1} e^{2 Y_{k \Delta}} \Delta   +  \rho   \sum_{k=0}^{n-1}  e^{Y_{k \Delta}} \Delta_k V  + \sqrt{1-\rho^2} 
 \sum_{k=0}^{n-1}  e^{Y_{k \Delta}} \Delta_k W, 
\end{align}
and the trapezoidal scheme
\begin{align}
\label{Trapez}
X_n^{{\sf Tr}} &= - \frac{1}{4} \sum_{k=0}^{n-1} \left( e^{2 Y_{k \Delta}} + e^{2 Y_{(k+1) \Delta}} \right) \Delta  +  
\rho   \sum_{k=0}^{n-1}  e^{Y_{k \Delta}} \Delta_k V
+ \frac{1}{2 }  \sqrt{1-\rho^2}  \sum_{k=0}^{n-1} \left( e^{ Y_{k \Delta}} + e^{Y_{(k+1) \Delta}} \right) \Delta_k W,
\end{align}
where $\Delta =T/n$ and $$ \Delta_k V= V_{(k+1)\Delta} -V_{ k \Delta}, \quad \Delta_k W= W_{(k+1)\Delta} -W_{ k \Delta}, \qquad k=0, \ldots, n-1.$$
For these schemes we have
$$ \lim_{n \rightarrow \infty} \, \left(\frac{n}{T}\right)^{2H}    \EX|X_T-  X_n^{\sf E}|^2    =  \frac{ 1 }{ 2H+ 1}   T \theta^2  {\bf E}|e^{Y_0}|^2 ,
$$
and
$$ \lim_{n \rightarrow \infty} \, \left(\frac{n}{T}\right)^{2H}   \EX|X_T-  X_n^{\sf Tr}|^2  = \left(\frac{1}{2H+1}- \frac{1-\rho^2}{4} \right)    T \theta^2 {\bf E}|e^{Y_0}|^2 , $$
see Theorem \ref{upper_b}.
Note that
$$ {\bf E}|e^{Y_t}-e^{Y_s}|^2 = \theta^2 {\bf E} |e^{Y_0}|^2 \cdot |t-s|^{2H} + o(|t-s|^{2H}) \qquad \textrm{for} \quad  |t-s| \rightarrow 0,$$
i.e. the limiting constants on the right hand side of the above expressions depend on the H\"older constant of the mean square smoothness of the volatility process $ \{e^{Y_t},\, t \geq 0\}$ .

The remainder of the manuscript is structured as follows. In the next section, we collect several properties of the stationary fractional Ornstein-Uhlenbeck process and other auxiliary results. Section 3 is devoted to
 the analysis of the Euler and the trapezoidal method, while in Section 4 we establish the lower bound for the minimal errors. Finally, in Section 5, we discuss the 
 joint
 simulation of the fractional 
Ornstein-Uhlenbeck process and Brownian motion.

\bigskip

\begin{remark} \cite{BFG} proposed that the correlation between $B$ and $V$ is introduced using the 
Mandelbrot--van Ness representation of fBm (\cite{MvN}), i.e. $V$ is in fact a two-sided Brownian motion $V=\{V_t, t \in \mathbb{R} \}$ and
$$ B_t= G(H)\int_{\mathbb{R}} ( (t-s)_+^{H-1/2}-(-s)_+^{H-1/2}) \dif V_s, \qquad t \in \mathbb{R},$$ with
$$ G(H)^2 = \frac{ 2H \, \Gamma(3/2-H)}{\Gamma(H+1/2) \Gamma(2-2H)}. $$ This leads to a correlation structure with
$$ \gamma(t,s) =\frac{G(H)}{H+1/2}\left( t^{H+1/2}- (t- \min \{t,s \})^{H+1/2} \right) {\mathbf 1}_{[0,\infty) \times [0, \infty)}(t,s), \qquad t,s \geq 0.$$ 
\end{remark}

\bigskip

\begin{remark}
  For It\=o stochastic differential equations (SDEs) driven by Brownian motion the minimal errors $e(n)$ have been  studied in detail, also for more general discretizations of the
Brownian motion, see \cite{mg} for a survey of the respective results. For stochastic differential equations driven by fractional Brownian motion minimal errors
have been studied for $H>1/2$ in \cite{spa} for the scalar case, respectively in \cite{levy} for the fractional L\'evy area.
\end{remark}

\bigskip



\begin{remark} In Section 5 we will point out that simulating  $$V_0, V_{T/n}, \ldots, V_T,\, \, \,
W_0, W_{T/n}, \ldots, W_T, \, \, \,  Y_{0},Y_{T/n}, \ldots, Y_T$$
exactly has a computational cost (number of random numbers and number of arithmetic operations) of order $n^2$, up to the best of our knowledge. This makes the  barrier of order $H$
even worse. In our future work, we will therefore study the approximation of $Y$ and $B$ via the Mandelbrot-van Ness representation of fBm similar to the recent work of \cite{Pak} 
and  also weak approximation methods.

\end{remark}

\section{Preliminaries}

\subsection{Fractional Brownian motion} Let $(\Omega, \mathcal{F}, \mathbf{P})$ be a probability space and recall that
a centered Gaussian process $B=\{B_t, t\in \mathbb{R} \}$ is called a fractional Brownian motion with Hurst index $H\in (0,1)$ if its covariance function equals
$$ K(s,t)=\frac 1 2 (t^{2H}+s^{2H}-|t-s|^{2H}), \qquad s,t \in \mathbb{R}.$$

Since $\EX|B_t-B_s|^2=|t-s|^{2H}$ there exists a modification with $\gamma$-H\"older trajectories for any $\gamma<H$, which we consider in what follows.
The process $B^H$ is moreover $H$-self-similar, i.e. for all $c>0$  we have
$$ \big \{c^{-H}B_{ct}, t\in  \mathbb{R} \big \}\stackrel{\mathsf{law}}{=}\{B_t, t \in \mathbb{R}\},$$
and is shift invariant, i.e. for any $s \in \mathbb{R}$ the process
$\{B_{t+s}-B_s, t\in  \mathbb{R} \}$ is a again an fBm.
Furthermore, fBm has polynomial growth as $|t| \rightarrow \infty$, i.e.  there exists a set $\mathcal{A} \in \mathcal{F}$ with $P(A)=1$ and a random variable $K$ such that
\begin{align}\label{subexp}
 |B_t(\omega)| \leq K(\omega)(1 + |t|^2), \qquad t \in \mathbb{R},\quad \omega \in \mathcal{A},
\end{align}
see  \cite{MS}.
In the sequel we will change
$\Omega$ such that $B_{\cdot}(\omega)=0$ for $\omega \notin \mathcal{A}$.

\subsection{Young integration}

Let $\alpha,\beta \in (0,1)$ such that $\alpha +\beta >1$ and  consider two H\"older functions $f\in \mathcal C^{\alpha}([0,T]; \mathbb{R})$, $g\in \mathcal C^{\beta}([0,T]; \mathbb{R})$. 
Then the Young integral $\int_0^T f(t)\dif g(t)$ is defined as the limit of the corresponding Riemann-Stieltjes sums, see e.g. \cite{young}.

Therefore, due to the H\"older smoothness of the sample paths  of fBm the  pathwise Riemann-Stieltjes integrals
$$ \int_s^t a(\tau) \, dB_{\tau}, \qquad -T \leq s <t \leq T,$$
exist if $a \in C^{\alpha}([-T,T]; \mathbb{R})$ with $\alpha +H > 1$. Moreover, if
$-T \leq s_1 <t_1 \leq s_2 <t_2 <T$,  we have the fractional It\=o isometry
\begin{align}
\label{frac_iso_2}
  \mathbf{E} \int_{s_1}^{t_1} a(\tau_1) \, dB_{\tau_1} \int_{s_2}^{t_2} a(\tau_2) \, dB_{\tau_2} = H(2H-1)
\int_{s_1}^{t_1} \int_{s_2}^{t_2} a(\tau_1)a(\tau_2) |\tau_1 -\tau_2|^{2H-2} \, d\tau_2 d\tau_1.
\end{align}

Property \eqref{subexp} implies that the improper Riemann-Stieltjes integrals
$$ \int_{-\infty}^t e^{\lambda  s} \dif B_s(\omega), \qquad t \in \mathbb{R}, \quad \omega \in \Omega,$$
are well-defined and that the integration by parts relation
\begin{align}
 \label{ibp}  \int_{-\infty}^t e^{\lambda  s} \dif B_s(\omega)=e^{\lambda t} B_t(\omega) - \lambda    \int_{-\infty}^t e^{\lambda u} B_u(\omega) \dif u, \qquad t \in \mathbb{R}, \quad \omega \in \Omega,
\end{align}
holds.

Moreover, for these integrals the isometry \eqref{frac_iso} is still valid, see \cite{maejima}:
\begin{align}
\label{frac_iso}
  \mathbf{E} \int_{-\infty}^{t_1} e^{\lambda \tau_1} \, dB_{\tau_1} \int_{t_2}^{t_3} e^{ \lambda \tau_2}  \, dB_{\tau_2} = H(2H-1)
\int_{-\infty}^{t_1}   \int_{t_2}^{t_3}  e^{ \lambda (\tau_1 + \tau_2)} |\tau_1 -\tau_2|^{2H-2} \, d\tau_2 d\tau_1
\end{align}
for $ - \infty <t_1 \leq t_2 \leq t_3 < \infty$.

\subsection{It\=o integration}
Throughout this manuscript, we assume that
$(B_t,W_t,V_t)_{t \geq 0} $ are $(\mathcal{F}_t)_{t \geq 0}$-adapted, where this filtration is constructed from the canonical filtration by the usual extension procedure, see e.g. Chapter 2.7 in \cite{KS}, and
that $(B_t)_{t<0}$ is $\mathcal{F}_0$-measurable.
Consequently, we must have that
\begin{align} \label{corr_struc} \gamma(t,s)=\gamma(t,t) \qquad \textrm{for all} \quad s \geq t, \end{align}
since the adaptedness implies that $B_t$ and $V_{t+h}-V_t$ with  $h>0$ are independent, and in particular
\begin{align} \gamma(t,s)=0 \qquad \textrm{for all} \quad t \leq 0 \leq  s.  \end{align}

Under the above adaptedness assumption the stochastic integrals
$$ \int_0^t e^{Y_s} \dif V_s, \qquad \int_0^t e^{Y_s} \dif W_s, \qquad t \in [0,T],$$
are standard It\=o integrals and we can use all classical tools as the It\=o isometry, Burkholder-Davis-Gundy inequality etc. 

Furthermore, we have the following Lemma:
\begin{lemma} \label{lem_cond_help_2}
Let  $R=(R_t)_{t\in [0,T]}$ be a process, which is $(\mathcal{F}_t)_{t \in [0,T]}$ adapted, independent of $W=(W_t)_{t\in [0,T]}$, and has root mean-square smoothness of order $\alpha \in (0,1)$, i.e.
there exist  $C>0$, $\alpha \in (0,1)$, such that
$$  \left( \mathbf{E} |R_t-R_s|^2 \right)^{1/2} \leq C|t-s|^{\alpha}, \qquad s,t \in [0,T].$$ Moreover, let $t_k=kT/n,\,k=0,\ldots,n$,
and  $$\overline{W}^n_t = W_{t_k} + \left(\frac{nt}{T} -k \right) \left( W_{t_{k+1}} -   W_{t_k}  \right), \qquad t \in [t_k, t_{k+1}], \,\, k=0, \ldots, n-1. $$
Then, it holds
\begin{gather*}
\EX \left|\int_0^T R_s\dif W_s- \frac{1}{2} \sum_{k=0}^{N-1} (R_{t_k}+R_{t_{k+1}})\Delta_k W \right|^2=\sum_{k=0}^{n-1}\int_{t_k}^{t_{k+1}}\EX \left|R_t- \frac{1}{2} (R_{t_k} + R_{t_{k+1}})\right|^2  \dif t
\end{gather*}
and
\begin{gather*}
\EX \left|\int_0^T R_s\dif W_s- \int_0^T R_s \dif  \overline{W}^n_s \right|^2=\sum_{k=0}^{n-1}\int_{t_k}^{t_{k+1}}\EX \left|R_t- \frac{n}{T}\int_{t_k}^{t_{k+1}}R_s\dif s\right|^2  \dif t.
\end{gather*}

\end{lemma}
\begin{proof}
To simplify our notation we put $T=1$. We only proof the second assertion, the other proof is similar.
First note that
$$ \int_{t_k}^{t_{k+1}}  R_s \dif  \overline{W}^n_s =  n \Delta_k W \,  \int_{t_k}^{t_{k+1}}  R_s \dif s.  $$
Therefore we have
$$  \mathbf{E} \left( \int_{t_k}^{t_{k+1}}  R_s \dif  \overline{W}^n_s  \int_{t_l}^{t_{l+1}}  R_s \dif  \overline{W}^n_s  \right)= n^2
 \mathbf{E} \left( \int_{t_k}^{t_{k+1}}  R_s \dif  s \int_{t_l}^{t_{l+1}}  R_s \dif s  \right) \mathbf{E} \left( \Delta_k W  \Delta_l W    \right) =
0 $$
for $k \neq l$ by independence of $R$ and $W$. Moreover, using additionally the adaptedness of $R$ and the properties of the It\=o integral it holds
\begin{align*}
 \mathbf{E} \left( \Delta_k W \, \int_{t_k}^{t_{k+1}}  R_s \dif  s  \int_{t_l}^{t_{l+1}}  R_s \dif W_s   \right)
 = \mathbf{E} \left( \Delta_k W  \, \int_{t_k}^{t_{k+1}}  R_s \dif  s  \, \mathbf{E} \left( \int_{t_l}^{t_{l+1}}  R_s \dif W_s  \Big{|} \mathcal{F}_{t_{k+1}} \right) \right) =
0 
\end{align*}
and
\begin{align*}
 \mathbf{E} \left(  \int_{t_k}^{t_{k+1}}  R_s \dif  W_s \, \Delta_l W  \int_{t_l}^{t_{l+1}}  R_s \dif s   \right)
 = \mathbf{E} \left( \int_{t_k}^{t_{k+1}}  R_s \dif  W_s  \, \mathbf{E} \left( \Delta_l W \int_{t_l}^{t_{l+1}}  R_s \dif s  \Big{|} \mathcal{F}_{t_{k+1}} \right) \right) =
0 
\end{align*}
for $k < l$. So we end up with
$$  \mathbf{E} \left( \int_{t_k}^{t_{k+1}}  R_s \dif  W_s  -\int_{t_k}^{t_{k+1}}  R_s \dif  \overline{W}^n_s  \right) \left( \int_{t_l}^{t_{l+1}}  R_s \dif  W_s  - \int_{t_l}^{t_{l+1}}  R_s \dif  \overline{W}^n_s  \right)=0. $$
for $k \neq l$.

Hence it remains to study
\begin{align*}
& \mathbf{E} \left|\int_{t_k}^{t_{k+1}}  R_s \dif  W_s  -\int_{t_k}^{t_{k+1}}  R_s \dif  \overline{W}^n_s  \right|^2    \\
 & \qquad = \int_{t_k}^{t_{k+1}}  \mathbf{E}  |R_s|^2 \dif  s  + 
  n \, \mathbf{E} \left| \int_{t_k}^{t_{k+1}}  R_s \dif s \right|^2   - 2  n \, \mathbf{E}   \left( \int_{t_k}^{t_{k+1}}  R_s \dif  W_s   \int_{t_k}^{t_{k+1}} R_s \dif s \Delta_k W \right).
\end{align*}
For the last term consider the approximation  $\int_{t_k}^{t_{k+1}}R_s\dif W_s=\lp^2-\lim_{N\to \infty}S^{(k)}_N$ with $$ S^{(k)}_N=\sum_{l=0}^{N-1}R_{s_{l}^N}\Delta_{l^N}W,$$ where $\{s_l^N \}_{l=0}^N$ is a sequence of partitions of $[t_k,t_{k+1}]$ with meshsize going to zero
and $\Delta_{l^N}W= W_{s_{l+1}^N}-W_{s_l^N}$.
(The $\lp^2$-convergence holds due to the mean-square smoothness assumption.)
Then we have
\begin{align*}
& \mathbf{E}   \left( \int_{t_k}^{t_{k+1}}  R_s \dif  W_s   \int_{t_k}^{t_{k+1}} R_s \dif s \Delta_k W \right) \\ & \quad =
\lim_{N\to \infty}\sum_{l=0}^{N-1} \EX \left(R_{s_l^N} \int_{t_k}^{t_{k+1}} R _s \dif s \right)\EX(\Delta_{l^N} W \Delta_k W)=\lim_{N\to \infty}\sum_{l=0}^{N-1} \EX \left(R_{s_l^N}\int_{t_k}^{t_{k+1}}R_s\dif s \right)(s_{l+1}^N-s_l^N)
\\ & \quad = \int_{t_k}^{t_{k+1}}\EX \left(R_t\int_{t_k}^{t_{k+1}}R_s\dif s \right)\dif t.
\end{align*}
Hence we obtain
\begin{align*}
& \mathbf{E} \left|\int_{t_k}^{t_{k+1}}  R_s \dif  W_s  -\int_{t_k}^{t_{k+1}}  R_s \dif  \overline{W}^n_s  \right|^2    \\
 & \qquad = \int_{t_k}^{t_{k+1}}  \mathbf{E}  |R_s|^2 \dif  s  + 
  n \, \mathbf{E} \left| \int_{t_k}^{t_{k+1}}  R_s \dif s \right|^2   - 2  n \, \int_{t_k}^{t_{k+1}}\EX \left(R_t\int_{t_k}^{t_{k+1}}R_s\dif s \right)\dif t
\end{align*}
and summing over the subintervals yields the assertion.
\end{proof}

\subsection{Stationary fractional Ornstein-Uhlenbeck process}

As already mentioned
$$ Y_t =  \mu  +  \theta e^{-\lambda t }\int_{-\infty}^t e^{\lambda u} \dif B_u, \qquad t \in \mathbb{R},$$
is the stationary solution of the Langevin SDE
$$ \dif Y_t = \lambda( \mu -Y_t ) \dif t + \theta \dif B_t, $$ see e.g. \cite{maejima,mjg}.
The stationarity is a simple consequence of the shift invariance of fBm which gives in particular
\begin{align} \label{stationarity}
 \left(Y_t,Y_s \right)&= \mu + \left(  e^{-\lambda t }\int_{-\infty}^t e^{\lambda u} \dif B_u,  e^{-\lambda s }\int_{-\infty}^s e^{\lambda u} \dif B_u\right)
 \\ \nonumber  & \stackrel{\mathsf{law}}{=} \mu + \left(  e^{-\lambda t }\int_{-\infty}^t e^{\lambda u} \dif (B_{u-s} -B_{-s}),  e^{-\lambda s }\int_{-\infty}^s e^{\lambda u} \dif (B_{u-s} -B_{-s}) \right)
 \\ \nonumber & =  \mu + \left(  e^{-\lambda (t-s) }\int_{-\infty}^{t-s} e^{\lambda u} \dif B_{u},  \int_{-\infty}^0 e^{\lambda u} \dif B_{u}  \right) =  (Y_{t-s},Y_0)
 \end{align}
for any $s,t \in \mathbb{R}$.

The process $Y$ is Gaussian
with mean
\begin{align*} \EX  Y_0 = \mu \end{align*}
and covariance
\begin{align*} \EX (Y_s - \EX Y_s) (Y_t - \EX  Y_t) =  R_Y(|t-s|),  \qquad s,t \in \mathbb{R}, \end{align*}
where
\begin{align} \label{cov_cosh}
 R_Y(\tau) &=
 \theta^2 \left(  \frac{ \Gamma(2H+1)\cosh(\lambda \tau)}{2\lambda^{2H}} - H \int_0^{\tau} \cosh(\lambda (\tau -u)) u^{2H-1} \dif u \right), \quad \tau \geq 0.
\end{align}
In particular  we have $$  {\mathbf{V}} (Y_t)=R_Y(0) =  \theta^2   \frac{ \Gamma(2H+1)}{2\lambda^{2H}}, \qquad t \in \mathbb{R}.$$
For a derivation see the Appendix. Note that another representation of the covariance function in terms of the confluent hypergeometric function
$ {}_1F_2$ has been given in Proposition 4.1.2 in  \cite{schoechtel}, starting from the Fourier representation
$$ R_Y(\tau)= \theta^2 \frac{\Gamma(2H+1) \sin(\pi H)}{2 \pi} \int_{-\infty}^{\infty} e^{i \tau x} \frac{|x|^{1-2H}}{\lambda^2+x^2} \dif x, \quad \tau \geq 0. $$
By straightforward calculations using \eqref{cov_cosh} and $2\cosh(\lambda (\tau -u)) = e^{\lambda \tau} e^{-\lambda u } + e^{-\lambda \tau} e^{\lambda u }$ we have:
\begin{lemma}\label{cor_Y} We have $ R_Y \in \mathcal{C}^{\infty}((0,\infty); \mathbb{R})$
and
$$ R_Y(\tau)= R_Y(0) - \frac{\theta^2}{2} \tau^{2H} + o(\tau^{2H}), \qquad \tau \rightarrow 0. $$
\end{lemma}

The process
$$ Z_t^{(a)} = e^{a (Y_t-\mu)}, \qquad t \in \mathbb{R},$$
is again stationary with
mean
$$ \EX Z_0^{(a)} = \exp\left(\frac{a^2}{2} R_Y(0)\right) $$
and 
covariance, again in terms of $R_Y$,
\begin{align*}
 R_{Z^{(a)}}(|t-s|)=\exp(a^2 R_Y(0)) \left( \exp(a^2 R_Y(|t-s|) )- 1 \right) . \end{align*}
Here we have exploited that the moment generating function of a $d$-dimensional  Gaussian random  variable $\xi$ with mean $m$ and covariance matrix $C$ is given by
$$ {\bf E} e^{ \langle a, \xi \rangle} = \exp \left( \langle a, m \rangle + \frac{1}{2} \langle a, Ca \rangle \right), \qquad a \in \mathbb{R}^d.   $$

 We have the following asymptotic expansion:
\begin{lemma} \label{benhenni-help} Let $a \neq 0$. We have $ R_{Z^{(a)}} \in \mathcal{C}^{\infty}((0,\infty); \mathbb{R})$
and
$$ R_{Z^{(a)}}(\tau)=  c_0-c_1\tau^{2H}+o(\tau^{2H}), \qquad \tau\to 0,$$
with 
\begin{align*}
 c_0&=c_0(a,R_Y(0))= \exp\left(a^2 R_Y(0)\right) \left( \exp\left(a^2 R_Y(0)\right) -1\right), \\
 c_1&=c_1(a,\theta,R_Y(0))=  \frac{a^2\theta^2}{2} \exp\left(2 a^2 R_Y(0)\right).
 \end{align*}
\end{lemma}

The previous Lemma allows to use results from \cite{Benhenni} for the approximation of $ \int_0^T Z^{(a)}_s \dif s $. We need another Lemma, which follows again by straightforward computations:

\begin{lemma}\label{smoothess_Z} Let the notation of the Lemma \eqref{benhenni-help} prevail and $ s \leq u \leq  t$. We have
$$ \EX |Z^{(a)}_t -Z^{(a)}_s|^2 = 2c_1|t-s|^{2H} + o(|t-s|^{2H}),\qquad |t-s|\to 0,$$
and
\begin{align*}
&\EX \Big| Z^{(a)}_u - \frac{1}{2}(Z_s^{(a)}+Z_t^{(a)}) \Big|^2  =
c_1 \big(|t-u|^{2H} +|s-u|^{2H}-\frac 1 2 |t-s|^{2H} \big)+o(|t-s|^{2H}),\qquad |t-s| \to 0.
\end{align*}
It also holds
\begin{align*}
&\EX \big( Z^{(a)}_t - Z_u^{(a)}\big) \big(Z_t^{(a)} -   Z_s^{(a)}\big) & = c_1 \big(|t-s|^{2H}+|t-u|^{2H} -|u-s|^{2H}\big) +o(|t-s|^{2H}),\qquad |t-s| \to 0.
\end{align*}
\end{lemma}

\bigskip
\section{Analysis of the Euler- and Trapezoidal scheme}

Let $$X_T^{RS}= \int_0^T e^{2Y_s} \dif s =  e^{2 \mu} \int_0^T Z_s^{(2)} \dif s.$$
Since $Z^{(a)}$ is stationary and the covariance function is infinitely differentiable away from zero and admits the expansion given in Lemma \ref{benhenni-help},
one can apply Theorem 1 in \cite{Benhenni} to obtain
\begin{align} \label{RS-app} \EX \left| X_T^{RS} - \frac{1}{2} \sum_{k=0}^{n-1} \left( e^{2 Y_{k \Delta}} + e^{2 Y_{(k+1) \Delta}} \right) \Delta \right|^2 =O( \Delta^{1+2H}).
\end{align}
Since
$$\frac{1}{2} \sum_{k=0}^{n-1} \left( e^{2 Y_{k \Delta}} + e^{2 Y_{(k+1) \Delta}} \right) \Delta =   \sum_{k=0}^{n-1}  e^{2 Y_{k \Delta}} \Delta + \frac{1}{2} \left( e^{2 Y_{T}} - e^{2 Y_{0}} \right) \Delta $$ 
and $H<1/2$, we also have
$$ \EX \left| X_T^{RS} -  \sum_{k=0}^{n-1}  e^{2 Y_{k \Delta}} \Delta \right|^2 =O( \Delta^{1+2H}).
$$

\begin{theorem}\label{upper_b}
 Suppose $X^{\mathsf E}_n$ and $X^{\mathsf{Tr}}_n$, $n\geq 1$, are given by \eqref{Euler} and \eqref{Trapez}, respectively. It holds
\begin{align*}
\EX \left| X_T -X_n^{\mathsf{E} } \right|^2 &= C_{\sf E} \cdot n^{-2H} + o(n^{-2H}),\\
\EX \left| X_T -X_n^{\mathsf{Tr} } \right|^2 &= C_{\sf{Tr}} \cdot n^{-2H} + o(n^{-2H}),
\end{align*}
with \begin{align*}
C_{\sf E}&=C_{\sf E}(\mu, \lambda, \theta, H, T)=\frac{2 e^{2\mu}
c_1(1,\theta,R_Y(0)) T^{2H+1}}{2H+1},  \\ 
C_{\sf Tr}&=C_{\sf Tr}(\mu, \lambda, \theta, \rho,  H, T)=  C_{\sf E}- (1-\rho^2) \frac{e^{2\mu}c_1(1,\theta,R_Y(0))T^{2H+1}}{2}.
\end{align*}
\end{theorem}

Note that
$$  \theta^2 {\bf E} |e^{Y_0}|^2 =   \theta^2  e^{2 \mu}  {\bf E}|Z_0^{(1)}|^2=   \theta^2 e^{2 \mu } e^{2 R_Y(0)} = 2 e^{2\mu}
c_1(1,\theta,R_Y(0))  .$$

\begin{proof}
The Riemann integral part of $X_T$, i.e. $X^{RS}_T$, is considered above. Recalling $\Delta=T/n$, $Y_t=e^{\mu}Z^{(1)}_t$ and using the It\=o isometry and Lemma \ref{smoothess_Z} we  have
\begin{align*}
\EX \left|\int_0^T e^{Y_s}\dif V_s-\sum_{k=0}^{n-1} e^{Y_{k\Delta}}\Delta_k V\right|^2&=\EX\left|\sum_{k=0}^{n-1}\int_{k\Delta}^{(k+1)\Delta}(e^{Y_s}-e^{Y_{k\Delta}})\dif V_s\right|^2
\\ & =e^{2\mu} \sum_{k=0}^{n-1}\int_{k\Delta}^{(k+1)\Delta}\EX|Z^{(1)}_s-Z^{(1)}_{k\Delta}|^2\dif s
\\ & = 2c_1 e^{2 \mu} \sum_{k=0}^{n-1} \int_{k\Delta}^{(k+1)\Delta} \left( (s-k\Delta)^{2H} + o(|s-k\Delta|^{2H}) \right)\dif s\\ &=
\frac{2c_1e^{2\mu}}{2H+1}n\Delta^{2H+1}(1+o(1)), 
\end{align*} 
with $c_1=c_1(1,\theta,R_Y(0))$
and analogously $$\EX \left|\int_0^T e^{Y_s}\dif W_s-\sum_{k=0}^{n-1} e^{Y_{k\Delta}}\Delta_k W\right|^2 = \frac{2c_1e^{2\mu}}{2H+1}n\Delta^{2H+1}(1+o(1)). $$
Summing up both terms using the orthogonality of the stochastic integrals we end up with the statement for the Euler scheme.

For the analysis of the trapezoidal scheme we have that
\begin{gather*}
\EX \left|\int_0^T Z_s^{(1)}\dif W_s- \frac{1}{2} \sum_{k=0}^{N-1} (Z^{(1)}_{k \Delta}+Z^{(1)}_{(k+1) \Delta})\Delta_k W \right|^2=\sum_{k=0}^{n-1}\int_{k \Delta}^{(k+1) \Delta}\EX \left|Z_t^{(1)}- \frac{1}{2} (Z_{k \Delta}^{(1)} + Z_{(k+1) \Delta}^{(1)})\right|^2  \dif t
\end{gather*}
by Lemma \ref{lem_cond_help_2}. Due to Lemma  \ref{smoothess_Z} we have
\begin{align*}
& \sum_{k=0}^{n-1}\int_{k \Delta}^{(k+1) \Delta}\EX \left|Z_t^{(1)}- \frac{1}{2} (Z_{k \Delta}^{(1)} + Z_{(k+1) \Delta}^{(1)})\right|^2  \dif t \\ & \qquad = c_1 \sum_{k=0}^{n-1}\int_{k \Delta}^{(k+1) \Delta} \left(  |t-k \Delta|^{2H}+|(k+1) \Delta-t|^{2H}-\frac 1 2 \Delta^{2H} +o(\Delta^{2H}) \right) \dif t
\\ & \qquad  =c_1 \left( \frac{2}{2H+1} -\frac{1}{2} \right)n\Delta^{2H+1}+o(\Delta^{2H}).
\end{align*}
Using this estimate, taking into account the correlation and the Euler estimate for the $\dif V$-integral, we obtain our assertion.

\end{proof}

\section{The order barrier}

\subsection{Conditional Expectations}
Let
\begin{align} \mathcal{G}_n = \sigma (V_{kT/n}, W_{kT/n},Y_{kT/n}, \, k=0, \ldots, n), \qquad \mathcal{H}_n= \sigma (V_{t},Y_t, t \geq 0, W_{kT/n}, \, k=0, \ldots, n)
\end{align}
and
\begin{align}
X_T^V= \rho \int_0^T e^{Y_s} \dif V_s, \qquad  \qquad
 X_T^W= \sqrt{1- \rho^2} \int_0^{T} e^{Y_s} \dif W_s.
\end{align}
We start with some representation formulae for the involved conditional expectations.

\begin{lemma} \label{lem_cond_help} (i) We have
 \begin{align}
  \EX (X_T^W | \mathcal{G}_n )&=  \sqrt{1- \rho^2} \int_{0}^T   \mathbf{E}\big( e^{Y_s} \big |  V_{kT/n},Y_{kT/n}, \, k=0, \ldots, n \big) \, d \overline{W}_s^n,
\label{cond-exp-XW}
 \end{align}
where $$\overline{W}^n_t = W_{kT/n} + \left(\frac{nt}{T} -k \right) \left( W_{(k+1)T/n} -   W_{kT/n}  \right), \qquad t \in [kT/n, (k+1)T/n], \,\,\,\,  k=0, \ldots, n-1, $$
 and
  \begin{align}
  \EX (X_T^W | \mathcal{H}_n )&=  \sqrt{1- \rho^2} \int_{0}^T   e^{Y_s}  \, d \overline{W}_s^n.
\label{cond-exp-XW}
 \end{align}
 
(ii) It holds 
 \begin{align*}
  \EX (X_T^V| \mathcal{G}_n )= \rho V_Te^{Y_T} - \rho \lim_{ \varepsilon \rightarrow 0 } \frac{1}{\varepsilon}    \mathbf{E}\Big(  \int_{0}^T V_s e^{{Y}_s^{\varepsilon}} ({Y}_s -   {Y}_{s-\varepsilon})
  \, d s \Big |  V_{kT/n},Y_{kT/n}, \, k=0, \ldots, n \Big) 
   \end{align*}
   with $${Y}_t^{\varepsilon} =  \frac{1}{\varepsilon}\int_{t-\varepsilon}^t Y_s \dif s, \qquad t \in \mathbb{R}, \quad 
   \varepsilon >0.$$
   \end{lemma}
\begin{proof}
The proof relies on the fact that for any sub-$\sigma$-algebra $\mathcal{G}$ of $\mathcal{F}$  and random variables $Z$, $Z_n$, $n \in \mathbb{N}$, we have  that
\begin{align} \label{conv_ce} {\small \lp^2-} \lim_{n \rightarrow \infty} Z_n =Z \quad \Longrightarrow  \quad {\small \lp^2}-\lim_{n \rightarrow \infty}  \mathbf{E} (Z_n | \mathcal{G})=  \mathbf{E} (Z | \mathcal{G}). \end{align}

(i) 
We start with the first equality. Consider the $\mathsf L^2$-approximation $X^W_T/ \sqrt{1-\rho^2} =\lim_{N\to\infty} S_N$ where
$$S_N=\sum_{l=0}^{N-1} e^{Y_{Tl/N}}(W_{(l+1)T/N}-W_{Tl/N}),\quad N\in \mathbb{N}.$$ Since $W$ is independent of $(V, Y)$, one can write 
$$ \EX(S_N| \mathcal G_n)= \sum_{l=0}^{N-1}\EX(e^{Y_{Tl/N}}|V_{kT/n},Y_{kT/n},k=0,\ldots,n)\EX(W_{(l+1)T/N}-W_{Tl/N}|W_{kT/n},k=0,\ldots,n).$$ Due to the 
normal correlation theorem we have 
$$\EX( W_{T(l+1)/N}-W_{Tl/N}|W_{kT/n}, k=0,\ldots,n)=\overline{W}^n_{T(l+1)/N} -\overline{W}^n_{Tl/N}$$ and hence 
$$
\EX(S_N| \mathcal G_n)=\sum_{l=0}^{N-1} \EX(e^{Y_{Tl/N}}|V_{kT/n},Y_{kT/n},k=0,\ldots,n)(\overline{W}^n_{(l+1)T/N} -\overline{W}^n_{Tl/N}).
$$
Since $\overline{W}^n_{\cdot}(\omega)$ is piecewise differentiable,  it follows
$$ \lim_{N \rightarrow \infty} \EX(S_N| \mathcal G_n) \stackrel{ {\bf P}-a.s.}{=}
 \int_0^T\EX(e^{Y_s}|V_{kT/n},Y_{kT/n},k=0,\ldots,n) \dif \overline{W}^n_s, $$
which finishes the proof of \eqref{cond-exp-XW} using \eqref{conv_ce}.

The second assertion can be shown analogously.

(ii)
To prove the second equality introduce the following family of random variables $$ X^{V,\varepsilon}_T=\rho \int_0^T e^{Y_t^{\varepsilon}}\dif V_t, \quad  \varepsilon>0.$$
The It\=o isometry, the mean value theorem and H\"older's inequality give that
$$     {\mathbf E} | X^{V,\varepsilon}_T  -    X^{V}_T|^2 = \int_0^T  {\mathbf E}| e^{Y_t^{\varepsilon}}-e^{Y_t} |^2 \dif t  \leq 
\int_0^T   \left( \mathbf{E} \left( \int_0^1  e^{ \lambda Y_t^{\varepsilon} +  (1-\lambda) Y_t}  \dif \lambda \right)^{4} \right)^{1/2}   \left( {\mathbf{E} |Y_t^{\varepsilon}}-Y_t|^4 \right)^{1/2} \dif t. 
$$
Since $$
|Y_t^{\varepsilon}| \leq \sup_{t \in [-1,T]} |Y_t| \qquad \textrm{for} \quad t \in [0,T], \,\, \varepsilon \in (0,1],$$
it follows that
 $$ {\mathbf E} | X^{V,\varepsilon}_T  -    X^{V}_T|^2  \leq C  \int_0^T \left( {\mathbf{E} |Y_t^{\varepsilon}}-Y_t|^4 \right)^{1/2} \dif t$$
 for $\varepsilon \in (0,1]$
 with $$C^2=    \mathbf{E}  e^{ 4 \sup_{t \in [-1,T]} | Y_t|}< \infty$$
 due to Fernique's theorem.
Finally we have
$$     \left( {\mathbf{E}} |Y_t^{\varepsilon}-Y_t|^4 \right)^{1/2} = c \, {\mathbf{E}} |Y_t^{\varepsilon}-Y_t|^2 \leq c \sup_{ s \in [t- \varepsilon,t]} {\mathbf{E}|Y_t-Y_s|^2}, \qquad t \in [0,T], $$
for some constant $c>0$ (again by Gaussianity of $(Y,Y^{\varepsilon})$, respectively the definition of $Y^{\varepsilon}$) and so Lemma \ref{cor_Y} implies that
$$ \lim_{\varepsilon \rightarrow 0} \int_0^T   \left( {\mathbf{E} |Y_t^{\varepsilon}}-Y_t|^4 \right)^{1/2} \dif t =0,$$
which shows that
$$     \lim_{\varepsilon \rightarrow 0}  {\mathbf E} | X^{V,\varepsilon}_T  -    X^{V}_T|^2 = 0.$$
Consequently, we have
\begin{gather*}
\EX( X^V_T| \mathcal G_n)=\lim_{\varepsilon \rightarrow 0 } \EX(X^{V,\varepsilon}_T| \mathcal G_n).
\end{gather*}

For all $\varepsilon>0$  the map
$ [0,T] \ni t \mapsto {Y}_t^{\varepsilon}(\omega) \in \mathbb{R}$ is differentiable, so partial integration gives
$$X^{V,\varepsilon}_T=\rho e^{Y_T^{\varepsilon}}V_T-\frac{\rho}{\varepsilon}\int_0^T V_t (Y_t-Y_{t-\varepsilon})e^{Y^\varepsilon_t}\dif t$$
and thus
$$ \EX( X^V_T| \mathcal G_n)
=\rho V_Te^{Y_T}-\rho\lim_{\varepsilon \rightarrow 0}\frac{1}{\varepsilon} \EX \left( \int_0^T V_t (Y_t -Y_{t-\varepsilon}) e^{Y^{\varepsilon}_t} \dif t \Big{|} V_{kT/n},Y_{kT/n},k=0,\ldots,n \right), $$
since $Y$ and $V$ are independent of $W$.
\end{proof}

The previous Lemma in particular implies that
\begin{align*}
  \EX (X_T^W | \mathcal{G}_n )&=  \sqrt{1- \rho^2} \sum_{l=0}^{n-1}  \frac{n}{T} \Delta_l W \int_{lT/n}^{(l+1)T/n}  \mathbf{E}\big( e^{Y_s} \big |  V_{kT/n},Y_{kT/n}, \, k=0, \ldots, n \big) \, \dif s 
 \end{align*}
and that $\EX (X_T^V | \mathcal{G}_n ) = \EX (X_T^V |  V_{kT/n},Y_{kT/n}, \, k=0, \ldots, n )$. Consequently we obtain that
$$\EX(\EX(X_T^W|\mathcal G_n)X_T^V)=\EX(\EX(X^V_T|\mathcal G_n)X^W_T)=\EX(\EX(X^V_T|\mathcal G_n)\EX(X^W_T|\mathcal G_n))=0,$$
since $(V,Y)$ and $W$ are independent. Since moreover  $\EX(X_T^WX^V_T)=0$ by independence of $W$ and $V$ it follows that
$$  \EX (X_T^W-  \EX (X_T^W | \mathcal{G}_n ))  (X_T^V-  \EX (X_T^V | \mathcal{G}_n )) = 0.$$
This yields
\begin{align} \label{lower_b_split}
\EX \left| X_T^W + X_T^V -  \EX (X_T^W+X_T^V | \mathcal{G}_n ) \right|^2 & =   \EX \left| X_T^W  -  \EX (X_T^W | \mathcal{G}_n ) \right|^2 +   \EX \left| X_T^V  -  \EX (X_T^V | \mathcal{G}_n ) \right|^2  \nonumber  \\ & \geq \EX \left| X_T^W  -  \EX (X_T^W | \mathcal{G}_n ) \right|^2, 
\end{align}
i.e. we can establish a lower bound for the minimal error by considering only the It\=o integral with respect to $W$.The optimal approximation of the $\dif V$-integral
seems to be much harder to analyse due the dependence of $Y$ and $V$.

After these preparations we can establish our lower error bound:

\begin{theorem}\label{lower_b}
In the notation above the following holds 
$$\liminf_{n\to \infty} \, n^{2H} \,\,  \EX(X_T-\EX(X_T|\mathcal G_n))^2 \geq  (1-\rho^2) \frac{2}{(2H+1)(2H+2)} T^{2H+1} e^{2 \mu}  c_1(1,\theta,R_Y(0)). $$
\end{theorem}
\begin{proof}
First note that
$$ \lim_{n \rightarrow \infty} n^{2H} \, \EX | X_T^{RS}-\EX(X_T^{RS} |\mathcal G_n)|^2=0$$
by \eqref{RS-app}.
Using this, \eqref{lower_b_split} and $ \mathcal{G}_n \subset \mathcal{H}_n$ it follows that 
$$ \liminf_{n\to \infty}n^{2H} \,  \EX |X_T-\EX(X_T|\mathcal{G}_n)|^2\geq \liminf_{n\to \infty}n^{2H} \, \EX|X^W_T-\EX(X^W_T|\mathcal{G}_n)|^2
\geq \liminf_{n\to \infty}n^{2H} \, \EX |X^W_T-\EX(X^W_T|\mathcal{H}_n)|^2.$$
The  Lemmata \ref{lem_cond_help}  and \ref{lem_cond_help_2} imply
$$ \EX |X^W_T-\EX(X^W_T|\mathcal{H}_n)|^2 = (1-\rho^2) \frac{n^2}{T^2} \sum_{k=0}^{n-1}\int_{kT/n}^{(k+1)T/n}\EX \left|\int_{kT/n}^{(k+1)T/n}(e^{Y_t}-e^{Y_s}) \dif s \right|^2\dif t.$$
Since
\begin{align*}
& \EX \left|\int_{kT/n}^{(k+1)T/n}(e^{Y_t}-e^{Y_s}) \dif s \right|^2
\\ \quad  & =  e^{2\mu}\int_{kT/n}^{(k+1)T/n} \int_{kT/n}^{(k+1)T/n} \EX (Z_t^{(1)} - Z_{s_1}^{(1)})  (Z_t^{(1)} - Z_{s_2}^{(1)}) \dif s_1 \dif s_2
\\  \quad &  = e^{2 \mu}c_1(1,\theta,R_Y(0)) \int_{kT/n}^{(k+1)T/n} \int_{kT/n}^{(k+1)T/n}  \left(  |t-s_1|^{2H}+|t-s_2|^{2H}-|s_1-s_2|^{2H} \right)\dif s_1 \dif s_2 + o(n^{-2H-2})
\end{align*}
by Lemma \ref{smoothess_Z}, it follows 
\begin{align*}
&  \frac{1}{c_1 e^{2 \mu} }\int_{kT/n}^{(k+1)T/n}\EX \left|\int_{kT/n}^{(k+1)T/n}(e^{Y_t}-e^{Y_s}) \dif s \right|^2\dif t
\\ & \quad =  \int_{kT/n}^{(k+1)T/n} \int_{kT/n}^{(k+1)T/n} \int_{kT/n}^{(k+1)T/n}  \left(  |t-s_1|^{2H}+|t-s_2|^{2H}-|s_1-s_2|^{2H} \right)\dif s_1 \dif s_2 \dif t \\ & \qquad + o(n^{-2H-3})
\\ & \quad =  \int_{0}^{T/n} \int_{0}^{T/n} \int_{0}^{T/n}  |s_1-s_2|^{2H} \dif s_1 \dif s_2 \dif t  + o(n^{-2H-3})
\\ & \quad = \frac{2}{(2H+1)(2H+2)}\Delta^{2H+3}  + o(n^{-2H-3}),
\end{align*} 
and summing up yields the assertion.
\end{proof}

\bigskip
\bigskip

\section{Joint Simulation of fOUp and Brownian motion}
Set $Y^c_t=Y_t-\mu,\, t\in \R$. Since integration by parts \eqref{ibp} gives
$$ Y^c_t = \theta \left( B_t - \lambda  e^{-\lambda t} \int_{-\infty}^{t} e^{\lambda s} B_u \dif u \right), $$
it follows that
$$ \mathbf{E}Y^c_t V_s = \theta \left(  \gamma(t,s) - \lambda e^{-\lambda t} \int_0^t e^{\lambda u} \gamma(u,s) \dif u \right), \qquad s,t \geq 0$$
In particular we have
\begin{align} \label{corr_Y_V}
\mathbf{E}Y^c_t (V_{s_2}-V_{s_1}) = \theta \left(  (\gamma(t,s_2)- \gamma(t,s_1)) - \lambda e^{-\lambda t} \int_0^t e^{\lambda u} (\gamma(u,s_2) - \gamma(u,s_1)) \dif u \right).
\end{align}

Hence for a given  $\gamma$ the covariance matrix $ C\in \mathbb{R}^{2n+1,2n+1}
$ of
$$ (Y^c, \Delta V)_{\sf disc}= (Y^c_{0}, Y^c_{T/n}, \ldots, Y^c_{T}, \,  \Delta_0 V, \Delta_1 V, \ldots, \Delta_{n-1} V),$$ 
can be computed explicitly and has the form
$$C =  \begin{pmatrix} C_{11}  &  C_{12} \\ 
 C_{21}
 &  C_{22} \end{pmatrix} = \begin{pmatrix} (\mathbf{E}  Y^c_{iT/n} Y^c_{jT/n})_{i,j =0, \ldots,n} &   (\mathbf{E} Y^c_{iT/n} \Delta_{j}V )_{i=0, \ldots,n, j=0, \ldots, n-1 }  \\ 
 ( \mathbf{E}  \Delta_i V   Y^c_{jT/n} )_{i=0, \ldots,n-1, j=0, \ldots, n }
 &  ( \mathbf{E} \Delta_i V \Delta_j V)_{i,j =0, \ldots,n-1}\end{pmatrix}.$$
Considering the increments  rather than  point evaluations of $V$  has the advantage that 
$$ C_{22}=(\mathbf{E} \Delta_i V \Delta_j V)_{i,j =0, \ldots,n-1} = {\Delta} \cdot I_{n},$$
where $I_{n}$ is the $n$-dimensional identity matrix.
Additionally, by
\eqref{corr_struc} and \eqref{corr_Y_V} we have that
$$ \mathbf{E} Y^c_{iT/n} \Delta_{j}V =0 \qquad \textrm{for} \quad i \leq j, $$
and therefore
$$  C_{12}=(\mathbf{E} Y^c_{iT/n} \Delta_{j}V )_{i=0, \ldots,n, j=0, \ldots, n-1 }=C_{21}'$$
is a lower  triangular matrix. Consequently, $C$ is a banded matrix.

We have
$$ (Y^c, \Delta V)_{\sf disc}' \stackrel{\mathsf{law}}{=}  L  \,  \left( \begin{array}{c}  \xi_1 \\ \xi_2 \\ \vdots \\  \xi_{2n+1} \end{array} \right), $$
where $L \in \mathbb{R}^{2n+1,2n+1}$ is the lower triangular matrix,  which arises from  the Cholesky decomposition of $C$, i.e. $C=LL'$, and  $\xi_i$, $i=1, \ldots, 2n+1$, are iid standard Gaussian random variables.
After precomputation of the matrix $L$, the computational cost (number of standard normal random numbers and arithmetic operations) to generate 
a sample of  $(Y^c, \Delta V)_{\sf disc}$ is $O(n^2)$. 

\medskip

Due to its stationarity  the random vector $(Y^c_0, Y^c_{T/n}, \ldots, Y^c_T)$  alone   can be sampled
 with a computational cost
of $O(n \log(n))$ using the Davis-Harte algorithm, see e.g. \cite{dh}. This method relies on the fact that after embedding $C_{11}$ in 
     a circulant matrix $C^{ce}$ of  size $m=2^{\lceil \log_2(n+1) \rceil +1}$ this matrix can be decomposed as $$ C^{ce}
= Q \Lambda Q^*, $$ where $\Lambda$ is the 
diagonal matrix of eigenvectors of $C^{ce}$, and $Q$ is the unitary matrix given by 
$$ Q_{j,k}
= \frac{1}{\sqrt{m}}  
\exp \left(  -2 \ci \pi \frac{jk}{m} \right), \qquad  j , k = 0, . . . , m-1. $$
Here $Q^*$ denotes the adjoint of $Q$, i.e. $(Q^*)_{j,k}$ is the complex conjugate of $Q_{k,j}$.
Moreover the eigenvalues of $C^{ce}$ are 
$$ \lambda_k =  \sum_{j=0}^{m-1} C^{ce}_{j,1} \exp \left(
2 \ci \pi \frac{ j k}{m} \right), \qquad k = 0, . . . , m-1.$$
Thus, we have
\begin{align}  Q \Lambda^{1/2} 
Q^* \xi  \, \stackrel{ \mathsf{law}}{=}  \,  \left( \begin{array}{c}  Y^c_0 \\ Y^c_{T/n} \\ \vdots \\  Y^c_{(m-1)T/n} \end{array} \right), \end{align}  where $\xi$ is a vector of $m$ independent standard normal random variables. 
The actual computation of the left hand side of the previous equation is carried out by using fast Fourier transformation, which leads to the computational cost
of $O(n \log(n))$.

Since 
\begin{align}  \sqrt{\Delta} Q I_n 
Q^* \xi  = \sqrt{\Delta} \xi \, \stackrel{ \mathsf{law}}{=}  \,  \left( \begin{array}{c}  \Delta_0 V \\ \Delta_1 V \\ \vdots \\  \Delta_{m-1}V, \end{array} \right), \end{align} 
it might be tempting to use 
$$ (Y^c, \Delta V)_{\sf disc}' \stackrel{\mathsf{law}}{=}    \left( \begin{array}{c}   Q \Lambda^{1/2} 
Q^* \xi \\  \sqrt{\Delta}  \xi \end{array} \right),      $$
which would have a computational cost of $O(n \log(n))$. However, this sampling procedure yields the covariance structure
$$ C_{12}= \sqrt{\Delta}  Q \Lambda^{1/2} 
Q^*, $$
which neither incorporates the covariance structure $\gamma$ nor is an upper triangular matrix, which is required by the adaptedness assumption for $B,V,W$.

\bigskip
\bigskip

\section{Appendix: Covariance function of the fOUp}
Recall that $Y_t^c=Y_t-\mu$, $t \in \mathbb{R}$.
By stationarity, i.e. \eqref{stationarity}, we only have to compute
 $ \mathbf{E} Y_0^cY_t^c$, $t \geq 0$.
 (i) First note that 
$$ \mathbf{E} Y_0^cY_t^c =  \mathbf{E} Y_0^c(Y_t^c-Y_0^c) + \mathbf{E} |Y_0^c|^2,  $$ 
and moreover
\begin{align}
\label{foup_calc_1} \mathbf{E} |Y_0^c|^2 = \theta^2 \frac{\Gamma(2H+1)}{2\lambda^{2H}},
\end{align}
see e.g. \cite{GJR}.

(ii) So it remains to consider $\mathbf{E}Y_0^c(Y_t^c-Y_0^c)$. Since \begin{align*} Y_t^c-Y_0^c =
\theta \exp(-\lambda t)   \int_{0}^t \exp(\lambda \tau) d B_{\tau} +
(\exp(-\lambda t)-1) Y_0^c
\end{align*}
the fractional It\=o isometry \eqref{frac_iso} now
gives
\begin{align*} \mathbf{E} Y_0^c(Y_t^c-Y_0^c) & = \theta^{2}
H(2H-1)\exp(-\lambda t)   \int_{0}^t \int_{-\infty}^0 \exp(\lambda
(\tau_1 + \tau_2)) |\tau_1 - \tau_2|^{2H-2} \dif  \tau_1 \dif  \tau_2 \\ &
\qquad + \theta^{2}(\exp(-\lambda t)-1)\frac{\Gamma(2H+1)}{2 \lambda^{2H}}.
\end{align*}
Hence, using step (i), we have
\begin{align} \mathbf{E} Y^c_0Y^c_t & = \theta^{2} \label{calc_cov_1}
H(2H-1)\exp(-\lambda t )   \int_{0}^t \int_{-\infty}^0 \exp(\lambda
(\tau_1 + \tau_2)) |\tau_1 - \tau_2|^{2H-2} \dif  \tau_1 \dif  \tau_2 \\ &
\nonumber \qquad + \theta^{2} \exp(-\lambda t )\frac{\Gamma(2H+1)}{2
\lambda^{2H}}.
\end{align}

\medskip

(iii) Using $\tau_1 = u + \tau_2$ we have
\begin{align*} \int_{0}^{t} \int_{-\infty}^0
\exp(\lambda (\tau_1 + \tau_2)) |\tau_1 - \tau_2|^{2H-2} \dif  \tau_1 \dif 
\tau_2 & =\int_{0}^{t} \int_{-\infty}^{-\tau_2} \exp(\lambda (u+ 2
\tau_2)) |u|^{2H-2} \dif  u  \dif \tau_2.
\end{align*}
Exchanging the order of integration  and using the transformation $u \mapsto -u$ we obtain
\begin{align*} & \int_{0}^{t} \int_{-\infty}^{-\tau_2} \exp(\lambda (u+ 2
\tau_2)) |u|^{2H-2} \dif  u  \dif \tau_2 \\ & \qquad = \int_{-\infty}^{-t} \int_{0}^{t}  \exp(\lambda (u+ 2
\tau_2)) |u|^{2H-2}  \dif \tau_2  \dif  u \\ & \qquad \quad +  \int_{-t}^{0} \int_0^{-u} \exp(\lambda (u+ 2
\tau_2)) |u|^{2H-2}  \dif \tau_2 \dif  u
\\
& \qquad  = \frac{1}{2 \lambda}  \left( \exp(2\lambda t ) -1 \right)  \int^{\infty}_{t}\exp(-\lambda u) |u|^{2H-2}   \dif  u
\\& \qquad \quad +  \frac{1}{2 \lambda}  \int^{t}_{0} \left( \exp(\lambda u) - \exp(-\lambda u) \right) |u|^{2H-2}   \dif  u 
\end{align*}
and therefore
\begin{align*} &
\exp(-\lambda t)   \int_{0}^t \int_{-\infty}^0 \exp(\lambda
(\tau_1 + \tau_2)) |\tau_1 - \tau_2|^{2H-2} \dif  \tau_1 \dif  \tau_2 \\ & \qquad =
\frac{1}{ \lambda}   \sinh(\lambda t)  \int^{\infty}_{t}\exp(-\lambda u) |u|^{2H-2}   \dif  u
\\& \qquad \quad +  \frac{1}{ \lambda}  \exp(-\lambda t) \int^{t}_{0}  \sinh(\lambda u)  |u|^{2H-2}   \dif  u 
\end{align*}

Integration by parts  gives
\begin{align*} &  \frac{2H-1}{\lambda} \sinh(\lambda t)  \int^{\infty}_{t}\exp(-\lambda u) |u|^{2H-2}   \dif  u
\\ & \qquad =   \sinh(\lambda t)  \int^{\infty}_{t}\exp(-\lambda u) |u|^{2H-1}   \dif  u - \frac{ \sinh( \lambda t) \exp(-\lambda t)}{\lambda}t^{2H-1}
\end{align*}
and
\begin{align*} &  \frac{2H-1}{\lambda}   \exp(-\lambda t) \int^{t}_{0} \sinh(\lambda u) |u|^{2H-2}   \dif  u
\\ & \qquad = \frac{1}{\lambda}  \exp(-\lambda t)\sinh(\lambda t) t^{2H-1}
\\ & \qquad \quad - \exp(-\lambda t) \int^{t}_{0}  \cosh(\lambda u)
|u|^{2H-1}\, \dif u, 
\end{align*}
since $H<1/2$.
Thus, we have
\begin{align*} & H(2H-1)
\exp(-\lambda t)   \int_{0}^t \int_{-\infty}^0 \exp(\lambda
(\tau_1 + \tau_2)) |\tau_1 - \tau_2|^{2H-2} \dif  \tau_1 \dif  \tau_2 \\ & \qquad =  H\sinh(\lambda t)  \int^{\infty}_{t}\exp(-\lambda u) |u|^{2H-1}   \dif  u -
 H\exp(-\lambda t) \int^{t}_{0}  \cosh(\lambda u)
|u|^{2H-1}\, \dif u.
\end{align*}

So, using  \begin{align} \label{gamma} \int_0^\infty t^b \exp(-at) \,\dif t = \frac{\Gamma(b+1)}{a^{b+1}} \end{align}
for $b>-1$, $a>0$ it follows 
\begin{align*}
& H(2H-1)
\exp(-\lambda t)   \int_{0}^t \int_{-\infty}^0 \exp(\lambda
(\tau_1 + \tau_2)) |\tau_1 - \tau_2|^{2H-2} \dif  \tau_1 \dif  \tau_2 \\ 
& \qquad =
  \sinh(\lambda t) \left( \frac{\Gamma(2H+1)}{2\lambda^{2H}}- H \int_0^{t}\exp(-\lambda u)|u|^{2H-1}   \dif  u \right)
\\& \qquad \quad - H \exp(-\lambda t)
\int^{t}_{0} \cosh(\lambda u) |u|^{2H-1}   \dif  u.
\end{align*}
Plugging this into \eqref{calc_cov_1} we have derived that
\begin{align} \mathbf{E} Y_0^cY_t^c& =  \theta^2 \cosh(\lambda t)\frac{\Gamma(2H+1)}{2
\lambda^{2H}} -  H  \theta^2
\int^{t}_{0} \cosh(\lambda (t-u))  |u|^{2H-1} \dif u.
\end{align}

\bibliographystyle{elsarticle-harv}
\bibliography{order-barrier}
\end{document}